\documentclass{amsart}
\usepackage{amssymb}
\usepackage{mathrsfs}
\usepackage{hyperref}
\usepackage{color}

\date{\today}

\newcommand{\Z}{{\mathbb Z}}
\newcommand{\R}{{\mathbb R}}

\newcommand{\be}{\begin{equation}}
\newcommand{\ee}{\end{equation}}

\newcommand{\ol}{\overline}

\newtheorem{theorem}{Theorem} [section]

\newtheorem{lemma}[theorem]{Lemma}

\newtheorem{definition}[theorem]{Definition}

\numberwithin{equation}{section}

\begin{document}

\title{SOME SPECTRUM  PROPERTY  OF  PERIODIC  COUPLING  AMO  OPERATOR}

\author[X.\ Xia]{Xu Xia}

\author[ZuoHuan  Zheng ]{ZuoHuan  Zheng}

\address{University  of Chinese  Academy of Sciences ，Beijing 100049,China}

\address{Academy of Mathematics and Systems Sciences  ,Chinese Academy of Sciences ，Beijing  100190 ,China}
\email{\href{mailto:helge.krueger@rice.edu}{xiaxu14@mails.ucas.ac.cn}}

\address{University  of Chinese  Academy of Sciences ，Beijing 100049,China}

\address{Academy of Mathematics and Systems Sciences  ,Chinese Academy of Sciences ，Beijing  100190 ,China}
\email{\href{mailto:helge.krueger@rice.edu}{zhzheng@amt.ac.cn}}

\thanks{X.\ X.\ was financial supported by NSF of China(No.11671382),CAS Key Project of Frontier  Sciences(No.QYZDJ-SSW-JSC003),the Key Lab of Random  Complex  Structures and  Data Sciences CAS and National Center for   Mathematics  and  Interdisplinary  Sciences  CAS.}
\thanks{Z.\ Z.\ was financial supported by NSF of China(No.11671382),CAS Key Project of Frontier  Sciences(No.QYZDJ-SSW-JSC003),the Key Lab of Random  Complex  Structures and  Data Sciences CAS and National Center for   Mathematics  and  Interdisplinary  Sciences  CAS.}

\date{\today}


\begin{abstract}
We study spectrum of the periodic coupling  AMO  model. Meantime there establish
the continuity of  Lyapunov exponent about the the periodic coupling of AMO  model. Through the dynamical method  can find a interval the AMO model only have  absolutely continuous spectrum. At the same time,
 some condition make the periodic coupling of AMO  model  is  singular continuous.
\end{abstract}

\maketitle

\section{Introduction}

In this paper, we construct examples of ergodic Schr\"odinger operators
$H_{\omega}$, whose Lyapunov exponent $L(E)$ is continuous and have some intrigue property.
Since  Anderson \cite{anderson1958absence}introduce that  absence of diffusion in certain random lattices, random
can make the model localization, and AMO model is determine in some sense, so if want to  get more random in the AMO model is consider more demension.

Since E.~Dinaburg and Y.~G. Sinai\cite{dinaburg1975one} use KAM theory to prove the reducible of cocycle, many absolutely continuous spectrum's existence have be proved. Eliasson\cite{eliasson1992floquet}  developed   the method  and prove that full measure of the reducibility. So when get the Lyapunov exponent vanishes  in somewhere, through the KAM theory can prove purely absolutely continuous spectrum's existence  in a internal about zero.

Through the Herman\cite{herman1983methode} method can prove  Lyapunov exponent  is positive, so can get a condition about purely singular continuous spectrum under some condition.

The method is Gordon\cite{gordon1976point} method, many place  use it to  prove can not exist point spectrum. The power of Gordon method  is very strong, when talk about continuous potential, can get the generic singular continuous spectrum.

Some author such as Simon, Jitomirskaya consider the transition of pure point spectrum and absolutely continuous spectrum of AMO model. At  the last forty   years, many people  consider the coupling is constant and study the famous conjecture
'Ten Martini Problem ' and 'Dry Ten Martini Problem '. Recetently the 'Ten Martini Problem' was solved by Avila and Jitomirskaya\cite{avila2009ten}. And recently the "Dry Ten Martini Problem " was partly solved by  Avila, You and Zhou \cite{avila2009ten11}. The above authors  always consider the coupling is constant.  What property about the coupling is  not a constant in the AMO model?  There is only a few results. In this paper we consider the coupling is not a constant in the AMO model, concrete to study the periodic couplings in the AMO model.

In physical the coupling and potential function determine present the magnetic, like  in the Anderson model. If  consider it from the coupling, we can think the coupling is iid and the potential function is identity, and no trival iid make the spectrum is Anderson localization then the material is not conductive. If in the AMO model the coupling is constant but if the absolutely of coupling is large than one, then under  some condition the material is not conductive. And when the absolutely of coupling is large than one is little than one but not zero,   the material will can not be conductive. This show that  the coupling is very important in the conductive of the material.
But the magnetic have many form, so the coupling maybe  not  is  a constant. There we will consider the non-constant coupling, concrete we study the periodic of the AMO model.

Since  1980 many author study sch\"odinger operator through a family of ergodic opertor, such as   Simon \cite{simon1982almost}, Sinai\cite{sinai1987anderson}, Avron and Simon\cite{avron1982singular}. The famous model AMO is define in a irrational rotation on circle. Throgh the frame of  strict ergodicity and transformation of the circle  can get many information about the spectrum. So in order to study the spectrum of the periodic coupling of AMO  model  we need to establish the dynamical  system. we get it in the second section.

Through the dynamical method and Lyapunov exponent there will be get some interesting spectrum  property of the   periodic coupling  AMO  model.
Find a property of that  periodic coupling  AMO  model has purely  absolutely  continuous spectrum on the neighborhood of zero (Theorem\ref{11}). Meantime give some singular  continuous spectrum  of periodic coupling  AMO  model.

In section  \ref{n0} show the motivation to consider the   periodic coupling  AMO  model, and give the definition about the Lyapunov exponent.
In section  \ref{n1}   give the continuity of  Lyapunov exponent of periodic coupling  AMO  model. It is important about to  use the Kotani theory.
In section \ref{n2} give the spectral property of a class of periodic coupling  AMO  model.
In section \ref{n3}   Find a property of that  periodic coupling  AMO  model has purely  absolutely  continuous spectrum on the neighborhood of zero.
In section \ref{n4}   studying a special case of periodic coupling  AMO  model, which give the connection with the classical AMO model.
In section \ref{n5}  give some singular  continuous spectrum  of periodic coupling  AMO  model.

\section{preliminaries}\label{n0}
First we consider the    periodic coupling  AMO  model is not want to get a extensions about classical AMO  model, but want get some spectrum property  of the  2-dimension AMO model.
In addition to being a result of interest of its own, we were motivated
by the following observation.
Studying how to get the Anderson-location in the 2-dimension,  this is the case when  I  consider the schodinger operator:
$$ H:L^2 (Z)\rightarrow L^2 (Z),$$
satisfy
$$(H\phi)_n=\phi_{n+1}+\phi_{n-1}+\lambda V(n\omega+\theta)\phi_n$$
$$\omega=(\omega_1,\omega_2)\in T^1\times T^1$$
$$\forall \phi\in L^2(Z).$$
We want to search the spectral property in  some sense  2-demention,  chose
$$V(n\omega+\theta)=\lambda(\cos(\theta+n\omega_1))+\beta(\cos(\theta+n\omega_2)),$$
the spectrum of  model equation becomes
$$\phi_{n+1}+\phi_{n-1}+\lambda(\cos(\theta+n\omega_1))+\beta(\cos(\theta+n\omega_2))
\phi_n=E \phi_n.$$

For simple  chose $$\lambda=\beta,$$
then
$$\lambda(\cos(\theta+n\omega_1))+\lambda(\cos(\theta+n\omega_2))=2 \lambda\cos(\theta+\frac{n}{2}(\omega_1+\omega_2))\cos(\frac{n}{2}(\omega_1-\omega_2)).$$

For simple there chose $$\frac{n}{2}(\omega_1-\omega_2)=\frac{p}{q}2\pi,$$
$p,q$ is co-prime and $p\in N,q\in Z$.
So we reduce the problem to the potential $v$ is a $$qusiperiodic \times periodic .$$
So for generally  chose $$V(n\omega+\theta)=\lambda\cos(n\omega+\theta)\times T(n),$$
where exist $k\in N$
$$T(n+k)=T(n)$$
$$\forall n\in Z.$$
So, if we know
$$
T(0),T(1),...,T(k-1),
$$
can get the potential. Through the Herman method there will can get some information about the positive of Lyapunov exponent.

We can get the transfer matrix through the above argument
\begin{equation}
\begin{split}
 A^{n}(\theta) & =\left(
            \begin{array}{cc}
              E-V(n\omega+\theta) & -1 \\
              1 & 0 \\
            \end{array}
          \right) \\
    &=\left(
             \begin{array}{cc}
               E-\lambda\cos(n\omega+\theta)\times T(n) & -1 \\
               1 & 0\\
             \end{array}
           \right),
\end{split}
\end{equation}
when  $$T=constant.$$ Then the model become  AMO  model
so we really consider bigger than AMO model.

Define
$$A_{n}(\theta)=\prod_{k=0}^{n-1}A^{k}(\theta) ;$$

$$ L(E)=\liminf_{n\rightarrow\infty}\int_{0}^{2\pi}\frac{\log|A_{n}(\theta)|}{n}d\theta.$$

Use the Herman  method   achieve that
\begin{equation}
\begin{split}
 A^{n}(\theta) &
          =\left(
             \begin{array}{cc}
               E-\lambda\cos(n\omega+\theta)\times T(n) & -1 \\
               1 & 0\\
             \end{array}
             \right) \\
    &  =\left(
             \begin{array}{cc}
               E-\frac{\lambda}{2}(e^i{(n\omega+\theta)\times T(n)}+e^{-i{(n\omega+\theta)\times T(n)}}) & -1 \\
               1 & 0\\
             \end{array}
             \right).
\end{split}
\end{equation}

Let$$e^{i(n\omega+\theta)}=z,$$
Then
 \begin{equation}
 \begin{split}
 A^{n}(z) &  =\left(
             \begin{array}{cc}
               E-\frac{\lambda}{2}(z\times T(n)+\frac{1}{z}\times T(n)) & -1 \\
               1 & 0\\
             \end{array}
             \right) \\
     &
             =\frac{1}{z}\left(
             \begin{array}{cc}
               E z-\frac{\lambda}{2}(z^2\times T(n)+1\times T(n)) & -z \\
              z & 0\\
             \end{array}
             \right);
 \end{split}
\end{equation}

\begin{equation*}
\begin{split}
 L(E) & =\liminf_{n\rightarrow\infty}\int_{0}^{2\pi}\frac{\log|A_{n}(\theta)|}{n}d\theta \\
    &  =\liminf_{n\rightarrow\infty}\int_{|z|=1}\frac{\log|A_{n}(z)|}{n}\frac{dz}{z}\\
    &\geq\int_{|z|=0}\frac{\log|A_{n}(0)|}{n}|dz|\\
    &= \log{\|\frac{\lambda}{2}|\sqrt[k]{T(1)\times T(2)...\times T(k)}}.
\end{split}
\end{equation*}

Establish $$ L(E)>0,$$ when
$$\|\frac{\lambda}{2}|\sqrt[k]{T(1)\times T(2)...\times T(k)}>1.
$$
so can get that the Lyapunov exponent is positive in some place.
It is interesting to search   the $$
\{T(0),T(1),...,T(k-1)\},
$$ have some one is zero, what happended  about  spectrum. Consider a special situation,
the  number of  set$$
\{T(0),T(1),...,T(k-1)\},
$$  is $2$.
And  $$
\{T(0), T(1)\},
$$ is
 $$
\{0, T(1)\},
$$
then $E=0$ is the Lyapunov exponent vanishes, no matter how large the  value of $T(1)$.

We will begin to discuss some basics dynamical systems
about  the periodic coupling of AMO  model.

Given a bounded sequence $V: \Z \to \R$, we denote by
$\Omega_V$ the hull of its translates. That is
\be
 \Omega_V = \ol{\{V_m,\quad m \in \Z\}}^{\ell^{\infty}(\Z)},
\ee
where $V_m(n) = V(n - m)$. If $\Omega_V$ is compact
in the $\ell^{\infty}(\Z)$ topology, then $V$ is
called {\em almost-periodic}.
The shift map on $\ell^\infty(\Z)$ becomes a translation
on the group $\Omega_V$ and it is uniquely
ergodic with respect to the Haar measure of $\Omega_V$.  Embedding a sch\"odinger operator in a
suit family is the main sutdy  method. This method get a great progress  in study {\em limit-periodic} by
Avila\cite{avila2009spectrum},he give a suit dynamical system on a contour group.
This re
$V$ is called {\em limit-periodic}, if there exists a sequence
of periodic potentials $V^k$ such that
\be
 V = \lim_{k \to \infty} V^k
\ee
in the $\ell^\infty(\Z)$ topology. It should be remarked
that limit-periodic $V$ are almost-periodic. In fact, then
$\Omega_V$ has the extra structure of being a Cantor group.
Then Avila confirm that limit-periodic one-to-one correspond a minimal transform Cantor group.

So use   the similar method to study the  periodic coupling  AMO  model.
The dyanmical system in study the periodic coupling  AMO  model is define in a compact metric space.

Let $$M=S^{1}\times \{
0,1,...,k-1
\}.$$
The hemeomorphism is
$$
T:M\rightarrow M,
$$

$$
T(\theta,h)=(\theta+\omega,h+1),
$$
and when
\begin{equation*}
h=k-1,
\end{equation*}
then
\begin{equation*}
h+1=0.
\end{equation*}

eg: the function in $$\{
0,1,...,k-1
\}$$  is cyclic group.

The topology in $M$ is product  topology, and we chose the topology in $\{
0,1,...,k-1
\}$  is discrete topology, the topology of $S^{1}$ is the general topology  induce by the metric.

The metric in  $S^{1}$
is
$$
d(\theta_1,\theta_2)=min(|\theta_1-\theta_2|,2\pi-|\theta_1-\theta_2|).
$$
Then  we can get the dynamical  system $(M,T)$ is strict ergodicity and transformation of the space$M$ when the number $\omega/{2\pi}$ is irrational.
This result is prove by many author, such as a book of F. R. Hertz, J. R. Hertz and R. Ures\cite{hertz2011partially}, and the unique probability mesaure is that:
$$d\mu=\frac{\frac{d\theta}{{2\pi}}\times\{\delta_{0}+\delta_{1}+...+\delta_{k-1}\}}{k}.$$

So can study the periodic coupling  AMO  model through the base dynamical system.

Since the base dynamical system is strict ergodicity, so the Lyapunov exponent can give  many information about  the periodic coupling  AMO  model.

First give the Lyapunov exponent in the model
let
\begin{equation*}
\forall \theta\in S^1,n\in \{
0,1,...,k-1
\},A(\theta,n)=\left(
             \begin{array}{cc}
               E-\lambda\cos(\theta)\times T(n) & -1 \\
               1 & 0\\
             \end{array}
           \right),
\end{equation*}

\begin{equation*}
\begin{split}
  A^{m}(\theta,n) &=\left(
             \begin{array}{cc}
               E-\lambda\cos(\theta+(m-1)\omega)\times T(n+m-1) & -1 \\
               1 & 0\\
             \end{array}
           \right) \cdots\\
            &\cdots\left(
             \begin{array}{cc}
               E-\lambda\cos(\theta+\omega)\times T(n+1) & -1 \\
               1 & 0\\
             \end{array}
           \right)
           \left(
             \begin{array}{cc}
               E-\lambda\cos(\theta)\times T(n) & -1 \\
               1 & 0\\
             \end{array}
           \right).
\end{split}
\end{equation*}

\begin{definition}
 Let $$L_m(E)=\int_{M} A^{m}(\theta,n)d\mu,$$ then through the subadditive ergodic theorem\cite{kingman1973subadditive} know that
 $$\lim_{m\rightarrow\infty}L_m(E),$$ is exist and $$\lim_{m\rightarrow\infty}L_m(E) =\inf_{m\geq 1} L_m(E),$$
 then $$L(E)=\lim_{m\rightarrow\infty}L_m(E) =\inf_{m\geq 1} L_m(E).$$
\end{definition}
The study  continuity of  Lyapunov exponent of periodic coupling  AMO  model is very important, through the Kotani theory\cite{kotani1986lyapunov}, if know the continuity of  Lyapunov exponent of periodic coupling  AMO  model, then the absolute continuous spectrum is in the essential closed of the  set $\{E|L(E)=0\}$. But if  Lyapunov exponent of periodic coupling  AMO  model is continuous, then set $\{E|L(E)=0\}$ is closed. So the essential closed of the  set $\{E|L(E)=0\}$ in $\{E|L(E)=0\}$. The continuity of  Lyapunov exponent  is not always can be  established, so the absolute continuous spectrum may have some support in the set$\{E|L(E)> 0\}$. This argument maybe construct many example to get discontinuous  Lyapunov exponent, such as Z.~Gan and H.~Krueger\cite{gan2010discontinuity}.

But the most important is that maybe get some $E$ in the  absolute continuous spectrum   meantime   in the  set $\{E|L(E)\geq0\}$, this maybe give a negative answer  about the "sch\"odinger conjecture". Avial\cite{avila2015kotani} give a counterexample  about the "sch\"odinger conjecture", his method is subtle, I think this argument can give a simple example. This problem will be done in the future.

\section{ continuity of  Lyapunov exponent of periodic coupling  AMO  model}\label{n1}
The study continuity of  Lyapunov exponent   AMO  model has long history, the first important result in this line is that J Bourgain's theorem:
\begin{theorem}\label{2}\cite{bourgain2002continuity}
The Lyapunov exponent $$L(\beta+.,.) : S^1 \times B^\omega(S^1, SL(2,R))\rightarrow R $$is
jointly continuous at every irrational $\beta$.
\end{theorem}
 In \cite{bourgain2002continuity}, Theorem \ref{2} was stated and proven for the Schrödinger, SL(2,R) case; strictly speaking, the  extension to SL(2,R)follows from \cite{jitomirskaya2009continuity}. Many generalization to non-singular and singular cocycles has been carried  out explicitly by\cite{jitomirskaya2012analytic}\cite{jitomirskaya2011continuity}.Many other line about the continuity of  Lyapunov exponent of different model, such as Duarte and  Klein\cite{duarte2016continuity} who prove many continuity of  Lyapunov exponent about general cocycle.
 But for prove the continuity of  Lyapunov exponent of periodic coupling  AMO  model, the theorem\ref{2} of  J Bourgain is enough. Meantime the method of  J Bourgain  can be use to prove the positive in many place. But in this paper only use the Herman and Avila's method to prove the positive.
\begin{theorem}
Given a  periodic coupling  AMO  model $(M,T)$, if the frequency $\omega$ is irrational, then the Lyapunov exponent is continuous in the real number.
\end{theorem}
\begin{proof}
Since  the dynamical  system $(M,T)$ is strict ergodicity and transformation of the space$M$ when the number $\omega/{2\pi}$ is irrational.
 the unique probability mesaure is that:
$$d\mu=\frac{\frac{d\theta}{{2\pi}}\times\{\delta_{0}+\delta_{1}+...+\delta_{k-1}\}}{k},$$
and
$$L(E)=\lim_{m\rightarrow\infty}L_m(E) =\inf_{m\geq 1} L_m(E).$$
Set $B(\theta)=A^{k}(\theta,0)$, then
let$$\bar{L}_m(E)=\int_{M}\frac{||B^m(\theta)||}{m}\frac{\frac{d\theta}{{2\pi}}}{k},$$
$$L(E)=\lim_{m\rightarrow\infty}\bar{L}_m(E) =\inf_{m\geq 1} \bar{L}_m(E).$$
Since the dynamical  system $(M,T)$ is strict ergodicity and transformation of the space  $M$  when the number $\omega/{2\pi}$ is irrational, the equation Set up.
If the  Lyapunov exponent  is defined by $$L(E)=\lim_{m\rightarrow\infty}\bar{L}_m(E) =\inf_{m\geq 1} \bar{L}_m(E).$$ Then the Lyapunov exponent can be reduced to a simple dynamical systems.
The new dynamical system is :
$$\bar{T}:S^1\rightarrow S^1;$$
$$\bar{T}(\theta)=\theta+k\omega;$$
$$B(\theta)=A^{k}(\theta,0).$$
Then this is a coycle on circle.
Thorough the definition of  $$B(\theta)=A^{k}(\theta,0),$$find that
\begin{equation*}
\begin{split}
 B(\theta)&=A^{k}(\theta,0) \\
    & =\left(
             \begin{array}{cc}
               E-\lambda\cos(\theta+(m-1)\omega)\times T(k-1) & -1 \\
               1 & 0\\
             \end{array}
           \right)
           \cdots\\
           &\cdots
           \left(
             \begin{array}{cc}
               E-\lambda\cos(\theta+\omega)\times T(1) & -1 \\
               1 & 0\\
             \end{array}
           \right)
           \left(
             \begin{array}{cc}
               E-\lambda\cos(\theta)\times T(0) & -1 \\
               1 & 0\\
             \end{array}
           \right).
\end{split}
\end{equation*}
So $$B(\theta)=A^{k}(\theta,0).$$
is analytical  about the variant $\theta\in S^1.$
and Then continuity of  Lyapunov exponent of periodic coupling  AMO  model reduced to the continuity of  Lyapunov exponent of  a simple dynamical systems$(S^1,\bar{T}).$
So $(B,T,S^1) \in S^1 \times B^\omega(S^1, SL(2,R))$, so use the theorem \ref{2}the Lyapunov exponent is continuous in the model $(B,T,S^1)$, and because periodic coupling  AMO  model reduced to the continuity of  Lyapunov exponent of  a simple dynamical systems$(S^1,\bar{T})$, so  Lyapunov exponent of periodic coupling  AMO  model is continuous.

\end{proof}

\section{ spectral property of a class of periodic coupling  AMO  model}\label{n2}

If we let the  number of  set$$
\{T(0),T(1),...,T(k-1)\}
$$  is $2.$

That is to say$$
\{T(0),T(1),...,T(k-1)\}
=
 \{T(0),T(1)\}.
$$
Then  the  model $(B,T,S^1)$  have some  simple form:
\begin{equation}
B(\theta)=\left(
\begin{array}{cc}
 (\cos (\theta ) T(0)-E) (\cos (\theta +\omega ) T(1)-E)-1 & E-\cos (\theta +\omega ) T(1) \\
 \cos (\theta ) T(0)-E & -1 \\
\end{array}
\right).
\end{equation}

\begin{theorem}
  Lyapunov exponent of periodic coupling  AMO  model,$L(0)=0,$when $T(0)=0,$ no matter what the value of $T(1)$.
\end{theorem}

\begin{proof}
\begin{equation}
B(\theta)=\left(
\begin{array}{cc}
 (\cos (\theta ) T(0)-E) (\cos (\theta +\omega ) T(1)-E)-1 & E-\cos (\theta +\omega ) T(1) \\
 \cos (\theta ) T(0)-E & -1 \\
\end{array}
\right).
\end{equation}
When $T(0)=0$, then
\begin{equation}
B(\theta)=\left(
\begin{array}{cc}
 -E (\cos (\theta +\omega ) T(1)-E)-1 &E-\cos (\theta +\omega ) T(1) \\
 -E & -1 \\
\end{array}
\right).
\end{equation}

When $E=0$,
\begin{equation}
B(\theta)=\left(
\begin{array}{cc}
 -1 & -\cos (\theta +\omega ) T(1) \\
 0 & -1 \\
\end{array}
\right).
\end{equation}
Since $\det B(\theta)=1$ so $L(E)\geq 0$.
 For  calculate the $L(0)$, use the Schmidt  norm to done it.
When $T(0)=0,E=0$, then
\begin{equation*}
\begin{split}
 || B(\theta)|| & =\sqrt{T(1)^2 \cos ^2(\theta +\omega )+2} \\
    & \leq\sqrt{T(1)^2+2},
\end{split}
\end{equation*}

\begin{equation*}
\begin{split}
  || B^m(\theta)|| & \leq\left(T(1)^2 \cos ^2(\theta +\omega )m+2\right)\\
    & \leq\left(T(1)^2 m+2\right).
\end{split}
\end{equation*}

So, can get the lower bound
\begin{equation*}
\begin{split}
 L(E) & =\lim_{m\rightarrow\infty}\int_{S^1}\frac{\ln|| B^m(\theta)||}{m}\frac{d\theta}{2\pi} \\
    & \leq\lim_{m\rightarrow\infty}\int_{S^1}\frac{\ln\left(T(1)^2 m+2\right)}{m}\frac{d\theta}{2\pi}.
\end{split}
\end{equation*}

Since
$$\lim_{m\rightarrow\infty}\frac{\ln\left(T(1)^2 m+2\right)}{m}=0,
$$
 so
$$
\lim_{m\rightarrow\infty}\int_{S^1}\frac{\ln\left(T(1)^2 m+2\right)}{m}\frac{d\theta}{2\pi}=0.
$$
Since $L(E)\geq 0$, can get $L(E)=0$, under the condition $T(0)=0,E=0.$
\end{proof}

Some interesting property can get in this periodic coupling  AMO  model. The classic AMO model is that:
$$ H:L^2 (Z)\rightarrow L^2 (Z),$$
satisfy
$$(H\phi)_n=\phi_{n+1}+\phi_{n-1}+\lambda V(n\omega+\theta)\phi_n.$$
Let $V(x)=\cos(x) $,the spectrum of  model equation becomes
$$\phi_{n+1}+\phi_{n-1}+\lambda \cos(n\omega+\theta)\phi_n=E \phi_n.$$

        And through the transform matrix to study the spectrum.
We know that Bourgain and Jitomirskaya result\cite{bourgain2002continuity} when the frequency $\omega$ is irrational, then $L(E)>0$ in the spectrum, if $|\lambda|>2$ and
$L(E)=0$, if $|\lambda|\leq2$, in the spectrum.

And through many people's work  get the complete  spectral picture, when  $|\lambda|>2$ if the irrational  frequency $\omega$ is well approximate by rational number, such as  Liouville  number ,then there only have the singular continuous spectrum.
when  $|\lambda|>2$ if the irrational  frequency $\omega$ is  diophantine number, then there only have the pure point spectrum.
when  $|\lambda|<2$ if the frequency $\omega$ is irrational, then  there only have absolutely continuous spectrum.   There are many gap in the above  example, use quantity of irrational  frequency $\omega$.
Let $\frac{p_n}{q_n}$ be the continued fraction approximation to $\omega$ and let
\begin{equation*}
    \beta =\limsup_{n\rightarrow\infty}\frac{\ln q_{n+1}}{q_n}.
\end{equation*}
In the paper \cite{avila2017sharp} have prove the conjecture that for alomst Mathieu family,under the condition that localization  for almost $a.e. x$ has been $e^\beta<\lambda$ where $\beta$ is the upper rate of exponential
growth of denominators of the continued fractions approximation to $\alpha$.

So many people wonder if every analytic potential have  these property. Then Avila study the neighborhood of $cos(\theta)$, and  Bjerkl\"ov \cite{bjerklov2006explicit} give explicit examples of arbitrarily large analytic ergodic potentials for which the Schrödinger equation has zero Lyapunov exponent for certain energies on matter what the value of the coupling.

In the two periodic of the  periodic coupling  AMO  model, if one of coupling is zero, then have get the similar property.

We can se some reason about the  property that zero Lyapunov exponent for certain energies if the first coupling is zero no  matter what the value of the second coupling.
\begin{equation}\label{4}
\begin{split}
 A(\theta,0) &=\left(
\begin{array}{cc}
 \cos (\theta ) T(0) & -1 \\
 1 & 0 \\
\end{array}
\right) \\
    & =\left(
\begin{array}{cc}
 0 & -1 \\
 1 & 0 \\
\end{array}
\right),
\end{split}
\end{equation}
\begin{equation}\label{5}
A(\theta+\omega,1)=
\left(
\begin{array}{cc}
 \cos (\theta +\omega ) T(1)-E & -1 \\
 1 & 0 \\
\end{array}
\right),
\end{equation}
\begin{equation}\label{6}
B(\theta)=A(\theta+\omega,1) A(\theta,0),
\end{equation}
then can see that
\begin{equation}\label{7}
\begin{split}
 A(\theta,0)& =
\left(
\begin{array}{cc}
 \cos \left(\frac{\pi }{2}\right) & -\sin \left(\frac{\pi }{2}\right) \\
 \sin \left(\frac{\pi }{2}\right) & \cos \left(\frac{\pi }{2}\right) \\
\end{array}
\right) \\
    & =R_{\frac{\pi}{2}}.
\end{split}
\end{equation}
So $R_{\frac{\pi}{2}}$ can exchange the contraction direction and expand direction, so the zero Lyapunov exponent for certain energies if the first coupling is zero no  matter what the value of the second coupling.
Such mechanism have give by  Furstenberg and Kesten\cite{furstenberg1960products} to and use it to establish  the example to explain his theorem is optimism. Bocker-Neto and Viana\cite{bocker2010continuity} use it to give the example to emphasize  the positive is important in the condition.

The  mechanism of the product the $SL(2,R)$ matrix if one is  $R_{\frac{\pi}{2}}$,then the contraction direction and expand direction can not be distinguished.So the Lyapunov exponent become zero.

So in the next section, we study the absolutely  continuous spectrum on the neighborhood of zero.

\section{absolutely  continuous spectrum on the neighborhood of zero}\label{n3}

Sinai\cite{dinaburg1975one} use KAM theory to prove the reducible of cocycle, many absolutely continuous spectrum's existence. But the notion reducible is  not a open condition.So many time use reducible only can prove the   exist the  absolutely  continuous spectrum, but can not prove the pure absolutely  continuous spectrum in  some place.

 Avila and  Jitomirskaya\cite{avila2010almost} introduce  a new notion almost reducible about the reducible of cocycle. Use this notion  they prove the  pure absolutely  continuous spectrum in AMO model, if the absolutely value of coupling is little than $2$, and frequency is  irrational.

 Resently X Hou, J You use the notion of almost reducibility and the equality of continuous dynamical system and discrete  cocycle they prove the almost reducible of AMO model if frequency is   Liouville number.

In the  progress in reduce a cocycle  we always incounter a small denominators  problem, this produce in look for a answer of coboundary  problem. This always be called the rigidity of irrational rotation.This paper not to get a sharp condition the interval of the exitance of  pure absolutely  continuous spectrum, the usual answer is enough to get a interval of zero.

If we let the  number of  set$$
T(0),T(1),...,T(k-1)
$$  is $2.$

That is to say$$
T(0),T(1),...,T(k-1)
=
 T(0),T(1)
$$
then  the  model $(B,T,S^1)$  have some  simple form:
\begin{equation}
B(\theta)=\left(
\begin{array}{cc}
 (\cos (\theta ) T(0)-E) (\cos (\theta +\omega ) T(1)-E)-1 & E-\cos (\theta +\omega ) T(1) \\
 \cos (\theta ) T(0)-E & -1 \\
\end{array}
\right).
\end{equation}

\begin{theorem}\label{9}
 There is a positive number $\epsilon>0$, the  sch\"odinger operator is almost reducible in the interval $[-\epsilon,\epsilon]$, if the frequency $\frac{\omega}{2\pi}$ is a diophantine number.
\end{theorem}

\begin{proof}
\begin{equation}
B(\theta)=\left(
\begin{array}{cc}
 (\cos (\theta ) T(0)-E) (\cos (\theta +\omega ) T(1)-E)-1 & E-\cos (\theta +\omega ) T(1) \\
 \cos (\theta ) T(0)-E & -1 \\
\end{array}
\right).
\end{equation}
When $T(0)=0$, then
\begin{equation}
B(\theta)=\left(
\begin{array}{cc}
 -E (\cos (\theta +\omega ) T(1)-E)-1 &E-\cos (\theta +\omega ) T(1) \\
 -E & -1 \\
\end{array}
\right).
\end{equation}

When $E=0$,
\begin{equation}
B(\theta)=\left(
\begin{array}{cc}
 -1 & -\cos (\theta +\omega ) T(1) \\
 0 & -1 \\
\end{array}
\right).
\end{equation}

Look for :
$$\cos (\theta +\omega ) T(1)=h(\theta +3\omega )-h(\theta +\omega ).$$
Since $\int_{S^1}\cos(\theta)d\theta=0$,and the frequency $\frac{\omega}{2\pi}$is a diophantine number.
the equation $$\cos (\theta +\omega ) T(1)=h(\theta +3\omega )-h(\theta +\omega ),$$ can be solve.

 This problem of Small denominators about the  mapping the circle onto itself deal with by the KAM method  used by Arnol'd\cite{arnol1961small} and Herman\cite{herman1979conjugaison}. And the result extended  by Yoccoz\cite{yoccoz1982conjugaison},  Sinai and  Khanin\cite{sinai1989smoothness}. The frequency has basic importance in the Small denominators. If the frequency is diophantine number then $h$ of the equation :
  $$\cos (\theta +\omega ) T(1)=h(\theta +3\omega )-h(\theta +\omega ),$$
 is exist and analytic.

 An  resent the problem of
 $$\cos (\theta +\omega ) T(1)=h(\theta +3\omega )-h(\theta +\omega ),$$
 has been extend that
the analytic radius $\rho$ of the function
 $$\cos (\theta +\omega ) T(1).$$

The irrational frequency $\omega$ have:
 Let $\frac{p_n}{q_n}$ be the continued fraction approximation to $\omega$ and let
\begin{equation*}
    \beta =\limsup_{n\rightarrow\infty}\frac{\ln q_{n+1}}{q_n}.
\end{equation*}
and  there is a inequality in the two number:
\begin{equation*}
    \beta<\rho.
\end{equation*}
Then there exist the analytic $h$ of the equation:
 $$\cos (\theta +\omega ) T(1)=h(\theta +3\omega )-h(\theta +\omega ).$$

So there can get more general result in the problem,  but only consider the diophantine number in
 this paper.

Then
\begin{equation}
\begin{split}
  \bar{B}(\theta)&=\left(
\begin{array}{cc}
 1 & -h(\theta +3 \omega ) \\
 0 & 1 \\
\end{array}
\right).\left(
\begin{array}{cc}
 -1 & T(1) (-\cos (\theta +\omega )) \\
 0 & -1 \\
\end{array}
\right).\left(
\begin{array}{cc}
 1 & h(\theta +\omega ) \\
 0 & 1 \\
\end{array}
\right) \\
    &=\left(
\begin{array}{cc}
 -1 & -h(\theta +\omega )+h(\theta +3 \omega )-\cos (\theta +\omega ) T(1) \\
 0 & -1 \\
\end{array}
\right).
\end{split}
\end{equation}

Because the equation $$\cos (\theta +\omega ) T(1)=h(\theta +3\omega )-h(\theta +\omega ),$$

\begin{equation*}
\bar{B}(\theta)=\left(
\begin{array}{cc}
 -1 & 0 \\
 0 & -1 \\
\end{array}
\right)=R_{\pi },
\end{equation*}
is a constant rotation, so it should in absolutely spectrum.
Next talk about the neighborhood of zero.

When $E\neq 0$,

\begin{equation}
\begin{split}
  B(\theta) &= \left(
\begin{array}{cc}
 -E (\cos (\theta +\omega ) T(1)-E)-1 & E-\cos (\theta +\omega ) T(1) \\
 -E & -1 \\
\end{array}
\right)\\
    & =\left(
\begin{array}{cc}
 -1 & -\cos (\theta +\omega ) T(1) \\
 0 & -1 \\
\end{array}
\right)+\left(
\begin{array}{cc}
 -E (\cos (\theta +\omega ) T(1)-E) & E \\
 -E & 0 \\
\end{array}
\right),
\end{split}
\end{equation}

\begin{equation*}
\begin{split}
  \bar{B}(\theta) & =\left(
\begin{array}{cc}
 1 & -h(\theta +3 \omega ) \\
 0 & 1 \\
\end{array}
\right).\left(
\begin{array}{cc}
 -E (\cos (\theta +\omega ) T(1)-E) & E \\
 -E & 0 \\
\end{array}
\right).\left(
\begin{array}{cc}
 1 & h(\theta +\omega ) \\
 0 & 1 \\
\end{array}
\right)
 \\
    & +\left(
\begin{array}{cc}
 1 & -h(\theta +3 \omega ) \\
 0 & 1 \\
\end{array}
\right).\left(
\begin{array}{cc}
 -1 & -\cos (\theta +\omega ) T(1) \\
 0 & -1 \\
\end{array}
\right).\left(
\begin{array}{cc}
 1 & h(\theta +\omega ) \\
 0 & 1 \\
\end{array}
\right)
 \\
    &=\left(
\begin{array}{cc}
 E h(\theta +3 \omega )-E (\cos (\theta +\omega ) T(1)-E) & G \\
 -E & -E h(\theta +\omega ) \\
\end{array}
\right)
\\
    &+\left(
\begin{array}{cc}
 -1 & -h(\theta +\omega )+h(\theta +3 \omega )-\cos (\theta +\omega ) T(1) \\
 0 & -1 \\
\end{array}
\right)
 \\
    &=\left(
\begin{array}{cc}
 E h(\theta +3 \omega )-E (\cos (\theta +\omega ) T(1)-E) & G \\
 -E & -E h(\theta +\omega ) \\
\end{array}
\right)
\\
    &+\left(
\begin{array}{cc}
 -1 &0 \\
 0 & -1 \\
\end{array}
\right),
\end{split}
\end{equation*}
where $G=h(\theta +\omega ) h(\theta +3 \omega ) E-h(\theta +\omega ) (\cos (\theta +\omega ) T(1)-E) E+E$.
The every entry of the matrix
\begin{equation*}
C(\theta)=
\left(
\begin{array}{cc}
 E h(\theta +3 \omega )-E (\cos (\theta +\omega ) T(1)-E) & G \\
 -E & -E h(\theta +\omega ) \\
\end{array}
\right),
\end{equation*}

is polynomial about $E$ and $h$, and the constant item is zero,
and $h$ is analytic in the condition,
so the matrix  $C(\theta)$ is approximate $\left(
\begin{array}{cc}
 0 & 0 \\
 0 & 0 \\
\end{array}
\right)$ is $o(E)$.

So
\begin{equation*}
\begin{split}
  \bar{B}(\theta) &  =\left(
\begin{array}{cc}
 -1 &0 \\
 0 & -1 \\
\end{array}
\right)+o(E) .\\
    &
\end{split}
\end{equation*}
The new dynamical system is :
$$\bar{T}:S^1\rightarrow S^1,$$
$$\bar{T}(\theta)=\theta+2\omega,$$
\begin{equation*}
\begin{split}
  \bar{B}(\theta) &  =\left(
\begin{array}{cc}
 -1 &0 \\
 0 & -1 \\
\end{array}
\right)+o(E). \\
    &
\end{split}
\end{equation*}
Then this is a coycle on circle.

And    $\bar{B}(\theta)$ is a approximate the constant matrix,but not arbitrary approximate the constant matrix,
through the theorem of  Hou and You \cite{hou2012almost}, there exist a $\epsilon>0$,

if
\begin{equation*}
o(E)<\epsilon,
\end{equation*}

$\bar{B}$    is almost reducible in the interval   $[-\epsilon,\epsilon].$

There is a theorem of  Hou and You\cite{hou2012almost} in the continuous case, the local
almost reducibility result is completely established recently, while
there is no result for global reducibility.

In the discrete case,  various global reducibility results \cite{avila2011kam}\cite{avila2006reducibility}\cite{avila2006reducibility} \cite{krikorian2001global}were
obtained, local almost reducibility results are not enough. Since they not deal with the  Liouville frequency.

Since the result of  Hou and You\cite{hou2012almost}, if we have a  connection about the  quasi-periodic cocycle close to constant and quasi-periodic linear system close to constant, they can get the almost reducibility result about Liouville frequency.

So,they prove a Embedding theorem to connection the  quasi-periodic cocycle close to constant and
quasi-periodic linear system close to constant.
\begin{theorem}\cite{hou2012almost}
 Any analytic quasi-periodic cocycle close to constant is the Poincaré map
of an analytic quasi-periodic linear system close to constant.
\end{theorem}
So the prove that
\begin{theorem}\cite{you2013embedding}\label{8}
 Any analytic quasi-periodic cocycle close to constant is the if the frequency is irrational then the
 analytic quasi-periodic cocycle have  almost reducibility .
\end{theorem}
The precise result can be ref \cite{you2013embedding}.

Continuous the proof of theorem \ref{9}, through the theorem\ref{8},
there is a positive number $\epsilon>0$, then $o(E)<\epsilon$,that  cocycle close to constant.
Then the   sch\"odinger operator is almost reducible in the interval $[-\epsilon,\epsilon]$.

\end{proof}

The almost reducibility of periodic coupling  AMO  model on the neighborhood of zero has been establish in the theorem \ref{9}. So use the property of the almost reducibility  can get some spectral property of of periodic coupling  AMO  model on the neighborhood of zero.

Avila have get a result in his paper :

\begin{theorem} \label{10}
 If the transform matrix have almost reducibility in some interval,then the sch\"odinger operator  is
 absolutely in that  interval.
\end{theorem}

Then use this theorem\ref{10}, get the theorem:

\begin{theorem}\label{11}
 There is a positive number $\epsilon>0$, the  sch\"odinger operator is purely absolutely continuous  in the interval $[-\epsilon,\epsilon]$,if the frequency $\frac{\omega}{2\pi}$ is a diophantine number.
\end{theorem}
\begin{proof}
Through the the theorem\ref{9},There is a positive number $\epsilon>0$, the  sch\"odinger operator is almost reducible in the interval $[-\epsilon,\epsilon]$.

So, by the theorem\ref{11},sch\"odinger operator is purely absolutely continuous  in the interval $[-\epsilon,\epsilon]$.
\end{proof}
Although there only give the simplest result about the periodic coupling  AMO  model have the purely absolutely continuous  in a interval. Since the result of You and Zhou\cite{you2013embedding} is generally include some  Liouville frequency.
That is:
\begin{equation*}
    \beta =\limsup_{n\rightarrow\infty}\frac{\ln q_{n+1}}{q_n}.
\end{equation*}

When $\beta>0$,can control the analytic radius $\rho$ of  the potential $V$ have the condition:
 \begin{equation*}
     \rho>5\beta>0,
 \end{equation*}
then can get a $\epsilon$ such that the  sch\"odinger operator is purely absolutely continuous  in the interval $[-\epsilon,\epsilon]$.

In this time only consider the diophantine number. But can easily extended to large number by the above argument.

\section{ a special case of periodic coupling  AMO  model}\label{n4}

If we let the  number of  set$$
\{T(0),T(1),...,T(k-1)\},
$$  is $2.$

That is to say$$
\{T(0),T(1),...,T(k-1)\}
=
\{ T(0),T(1)\},
$$
and the case that  one of
\begin{equation*}
\{T(0),T(1)\},
\end{equation*}
is zero has been consider  in the previous section.
There consider the case:
\begin{equation*}
\{T(0),T(1)\}=\{1,-1\}.
\end{equation*}
The potetial is
\begin{equation*}
\begin{split}
  V(n\omega+\theta) & =\lambda\cos(n\omega+\theta)\times T(n) \\
    & =\lambda\cos(n\omega+\theta)\times (-1)^n\\
    & =\lambda\cos(n\omega+\theta+n\pi)\\
   & =\lambda\cos(n(\omega+\pi)+\theta).
\end{split}
\end{equation*}

The  periodic coupling  AMO  model  become a classic AMO medel:

$$\bar{T}:S^1\rightarrow S^1,$$
$$\bar{T}(\theta)=\theta+\omega+\pi,$$
\begin{equation*}
D(\theta)=\left(
             \begin{array}{cc}
               E-\lambda\cos(\theta) & -1 \\
               1 & 0\\
             \end{array}
           \right).
\end{equation*}
Then this is a coycle on circle. But have a variant in frequency in the classic AMO medel.
So in the condition in the  above, has a obvious theorem.
\begin{theorem} \label{12}
If consider the value of $\lambda$, there has:
\begin{itemize}
  \item If  $|\lambda|>2$ ,  the Lyapunov exponent $L(E)$of the cocycle $(D,T)$ is positive an get the value $\ln(\frac{|\lambda|}{2}), $if  $E$ in the spectrum of sch\"odinger operator,and the measure of the spectrum is $4|1-\frac{2}{|\lambda|}|$.

  \item If $|\lambda|=2$, the Lyapunov exponent $L(E)$ of the  cocycle $(D,T)$ is positive an get the value $0 , $ if $E$ in the spectrum of sch\"odinger operator, and the measure of the spectrum is $0$.
  \item If  $|\lambda|<2$,the Lyapunov exponent $L(E)$ of the cocycle $(D,T)$ is positive an get the value $0,$ if $E$ in the spectrum of sch\"odinger operator, and the measure of the spectrum is $4|1-\frac{|\lambda|}{2}|$.
\end{itemize}

\end{theorem}
Only in this special case can get the complete information about the Lyapunov exponent in the spectrum.
mean time there is a theorem about the spectrum of the special example. The reason is that the special AMO model
can be reduced to the classic model. So the spectrum of the special example only translate the spectrum of  the classic AMO to the the special example.

\section{  singular  continuous spectrum  of periodic coupling  AMO  model}\label{n5}

The explanation about the positive  of  Lyapunov exponent of periodic coupling  AMO  model has been given in the introduction. In this section give a deep  result about the positive  of  Lyapunov exponent of periodic coupling  AMO  model.

The Gordon \cite{gordon1976point} use the periodic approximate the model can exclude the point spectrum under some condition. This method has been used in ubiquitous about sch\"odinger operator.

 Avron and Simon \cite{avron1982singular} use the positive  of  Lyapunov exponent of   AMO  model  exclude the absolutely
continuous spectrum. And they use the method of Gorodon they prove that:
If frequency is a Liouville number and the coupling$\lambda > 2,$ we prove that for a.e. phase$\theta,$ the operator's spectral measures are all singular continuous.

 In the case of the potential is substitution Hamiltonians, Damanik \cite{damanik1998singular}
 consider discrete one-dimensional Schrödinger operators with potentials generated by primitive substitutions. A purely singular continuous spectrum with probability one is established provided that the potentials have a local four-block structure. The local four-block structure is the main tool of Gorodon method. So, the Gorodon method is
 strongly in prove the spectral type  of Schrödinger operators.

 The detail of Gorodon method can be find in  Damanik\cite{damanik1999gordon}. Damanik and Stolz use it prove a fully general
 result\cite{damanik2000generalization}.

 Although there is a method to exclude the point spectrum and absolutely spectrum, I wonder if there is a method prove singular continuous spectrum directly, and will be search it in the future.

The periodic coupling  AMO  model is similar classic AMO  model, so use the Gorodon method get a result about
the spectrum type of  periodic coupling  AMO  model.

For  simplify  there consider the  number of  set$$
\{T(0),T(1),...,T(k-1)\},
$$  is $2.$

That is to say$$
\{T(0),T(1),...,T(k-1)\}
=
 \{T(0),T(1)\}.
$$
\begin{definition}
A number $\alpha\in \mathbb{R/Q}$  is called  a Liouville  number, if for any $k\in \mathbb{N}$, there exist
$p_{k},q_{k}\in \mathbb{N}$ such that:
\begin{equation*}
|\alpha-\frac{p_{k}}{q_{k}}|\leq k^{-q_{k}}.
\end{equation*}
\end{definition}
The Liouville  numbers  have many interesting property in topology and measure.

The set of Liouville  numbers is small form an analyst's point of view :
It has Lebesgue  measure zero .

However, from a topologist's point of view,it is rather big :
It is a $G_{\delta}-$set, so it is generic in the topology sense.
For prove the theorem,there is a lemma in  Cycon,  Froese, Kirsch and Simon's book\cite{cycon2009schrodinger}.
\begin{lemma}\cite{cycon2009schrodinger}\label{13}
Let $A$ be an invertible $2\times 2$ matrix,and $v$ ia a vector of norm $1.$
Then
\begin{equation*}
\max(||A v||,||A^{2} v||,||A^{-1} v||,||A^{-2} v||)\geq \frac{1}{2}.
\end{equation*}
\end{lemma}

\begin{theorem}
 If the frequency $\frac{\omega}{2\pi}$ is a  Liouville  number , and the periodic coupling is
 $
 T(0),T(1)
 $
 and
$|T(0) T(1)|>4$,
then the periodic coupling  AMO  model is pure singular measure in the spectrum.
\end{theorem}
\begin{proof}
In  this case  there  only consider the one dimension, for fix $E$
consider the one-dimensional difference sch\"odinger  operator :
\begin{equation*}
u(n+1)+u(n-1)+[cos(\theta+n\omega)-E]u(n)=0.
\end{equation*}
For the strict ergodic of periodic coupling  AMO  model, only need  to consider the transform matrix:

\begin{equation*}
A(\theta )=\left(
\begin{array}{cc}
 E-T(0) \cos (\theta ) & -1 \\
 1 & 0 \\
\end{array}
\right),
\end{equation*}

\begin{equation*}
A(\theta +\omega )=\left(
\begin{array}{cc}
 E-T(1) \cos (\theta +\omega ) & -1 \\
 1 & 0 \\
\end{array}
\right),
\end{equation*}

\begin{equation*}
\begin{split}
  B(\theta ) & =A(\theta ) A(\theta +\omega ) \\
    & =\left(
\begin{array}{cc}
 E-T(1) \cos (\theta +\omega ) & -1 \\
 1 & 0 \\
\end{array}
\right)
    \cdot\left(
\begin{array}{cc}
 E-T(0) \cos (\theta ) & -1 \\
 1 & 0 \\
\end{array}
\right),
\end{split}
\end{equation*}

\begin{equation*}
A^{-1}(\theta )=\left(
\begin{array}{cc}
 0 & 1 \\
 -1 & E-T(0) \cos (\theta ) \\
\end{array}
\right),
\end{equation*}

\begin{equation*}
A^{-1}(\theta +\omega )=\left(
\begin{array}{cc}
 0 & 1 \\
 -1 & E-T(1) \cos (\theta +\omega ) \\
\end{array}
\right),
\end{equation*}

\begin{equation*}
\begin{split}
B^{-1}(\theta ) &=[A(\theta ) A(\theta +\omega )]^{-1} \\
    &  =A^{-1}(\theta +\omega )\cdot A^{-1}(\theta) \\
    &  =A^{-1}(\theta +\omega )\\
    &=\left(
\begin{array}{cc}
 0 & 1 \\
 -1 & e-T(1) \cos (\theta +\omega ) \\
\end{array}
\right)\cdot\left(
\begin{array}{cc}
 0 & 1 \\
 -1 & E-\cos (\theta ) T(0) \\
\end{array}
\right).
\end{split}
\end{equation*}

So there give the explicit  expression of the transform  matrix, next talk about the local block structure.
 Since have know the Lyapunov exponent is positive upon the condition   $|T(0) T(1)|>4$, and the continuity of
 Lyapunov exponent  of the periodic coupling  AMO  model has been proved. So the absolutely continuous spectrum  of the periodic coupling  AMO  model is empty. Then look for the like-Gorodon block structure.

 Assume that $\frac{\omega}{2\pi}$ is well approximated by $\frac{p_{k}}{q_{k}}$ in the sense of above definition.
 By choosing a subsequence $\frac{p_{k'}}{q_{k'}}$ of $\frac{p_{k}}{q_{k}}$, assumme

\begin{equation*}
 |\frac{\omega}{2\pi}-\frac{p_{k'}}{q_{k'}}|\leq q^{-1}_{k'}k^{-q_{k'}}.
\end{equation*}

Then set
\begin{equation*}
V_{k}(n)=T(n)\cos(2\pi\frac{p_{k'}}{q_{k'}}n+\theta).
\end{equation*}
Then $T_k=2q_{k'}$ is a periodic for $V_{k}$.
estamate:

\begin{equation}\label{16}
\begin{split}
  \sup_{|n|\leq 4q_{k'}}|V_{k}(n)-V(n)| & =\sup_{|n|\leq 4q_{k'}}|T(n)\cos(2\pi\frac{p_{k'}}{q_{k'}}n+\theta)-T(n)\cos(\omega n+\theta)| \\
    & \leq C\sup_{|n|\leq 4q_{k'}}2\pi|n||\frac{p_{k'}}{q_{k'}}-\omega|\\
    &\leq 4C\pi k^{-T_{k}}.
\end{split}
\end{equation}

So this give the accurate error of  periodic approximate.
Then use the $V_{k}(n)$ to approximate the periodic coupling  AMO  model.
Then use the the follwing  periodic  cocycle to  approximate the periodic coupling  AMO  model:

\begin{equation*}
A_{k}(\theta )=\left(
\begin{array}{cc}
 E-T(0) \cos (\theta ) & -1 \\
 1 & 0 \\
\end{array}
\right),
\end{equation*}

\begin{equation*}
A_{k}(\theta +2\pi\frac{p_{k'}}{q_{k'}} )=\left(
\begin{array}{cc}
 E-T(1) \cos (\theta +2\pi\frac{p_{k'}}{q_{k'}} ) & -1 \\
 1 & 0 \\
\end{array}
\right),
\end{equation*}

\begin{equation*}
\begin{split}
  B_{k}(\theta ) & =A_{k}(\theta ) A_{k}(\theta +2\pi\frac{p_{k'}}{q_{k'}} ) \\
    & =\left(
\begin{array}{cc}
 E-T(1) \cos (\theta +2\pi\frac{p_{k'}}{q_{k'}} ) & -1 \\
 1 & 0 \\
\end{array}
\right)
    \cdot\left(
\begin{array}{cc}
 E-T(0) \cos (\theta ) & -1 \\
 1 & 0 \\
\end{array}
\right),
\end{split}
\end{equation*}

\begin{equation*}
A_{k}^{-1}(\theta )=\left(
\begin{array}{cc}
 0 & 1 \\
 -1 & E-T(0) \cos (\theta ) \\
\end{array}
\right),
\end{equation*}

\begin{equation*}
A_{k}^{-1}(\theta +2\pi\frac{p_{k'}}{q_{k'}} )=\left(
\begin{array}{cc}
 0 & 1 \\
 -1 & E-T(1) \cos (\theta +2\pi\frac{p_{k'}}{q_{k'}} ) \\
\end{array}
\right),
\end{equation*}

\begin{equation*}
\begin{split}
B_{k}^{-1}(\theta ) &=[A_{k}(\theta ) A_{k}(\theta +2\pi\frac{p_{k'}}{q_{k'}} )]^{-1} \\
    &  =A_{k}^{-1}(\theta +2\pi\frac{p_{k'}}{q_{k'}} )\cdot A_{k}^{-1}(\theta) \\
    &  =A_{k}^{-1}(\theta +2\pi\frac{p_{k'}}{q_{k'}} )\\
    &=\left(
\begin{array}{cc}
 0 & 1 \\
 -1 & E-T(1) \cos (\theta +2\pi\frac{p_{k'}}{q_{k'}} ) \\
\end{array}
\right)\cdot\left(
\begin{array}{cc}
 0 & 1 \\
 -1 & E-\cos (\theta ) T(0) \\
\end{array}
\right).
\end{split}
\end{equation*}

Through the lemma\ref{13}, if $v$ ia a arbitrary vector of norm $1$, $\forall \theta\in S^{1}$
can get :

\begin{equation*}
\max(||B_{k}^{T_{k}}(\theta )
 v||,||B_{k}^{2 T_{k}}(\theta )
 v||,||B_{k}^{-T_{k}}(\theta )
 v||,||B_{k}^{-2T_{k}}(\theta )
 v||)\geq \frac{1}{2}.
\end{equation*}
Let $M$ is :
\begin{equation*}
M=\max\{|T(0)|,|T(1)|\}.
\end{equation*}
Then for $\forall \theta\in S^{1}$:

\begin{equation*}
A(\theta,n )=\left(
\begin{array}{cc}
 E-T(n) \cos (\theta ) & -1 \\
 1 & 0 \\
\end{array}
\right).
\end{equation*}
Since there only consider the $E$ in the spectrum, so
\begin{equation*}
|E|\leq M+2,
\end{equation*}
\begin{equation*}
\begin{split}
 ||A||^{2}_{Schmidt} & =(E-T(n) \cos (\theta ))^{2}+(-1)^{2}+(1)^{2}+(0)^{2} \\
    & =E^{2}-2T(n) \cos (\theta )E+T(n)^{2} \cos^{2} (\theta )+2\\
    &\leq  (M+2)^{2}+2M(M+2)+M^2+2.
\end{split}
\end{equation*}
so
\begin{equation*}
\begin{split}
 ||A||_{Schmidt} & \leq \sqrt{(M+2)^{2}+2M(M+2)+M^2+2)}\\
    & \leq 2(M+2).
\end{split}
\end{equation*}
Let
\begin{equation*}
\bar{M}=\max( 2(M+2),M).
\end{equation*}
The difference of the two cocycle is :

\begin{equation}\label{14}
\begin{split}
 &||B_{k}^{T_{k}}(\theta )-B^{T_{k}}(\theta )|| \\
  & =||A(\theta+4(T_{k}-1)\pi\frac{p_{k'}}{q_{k'}},T(2(T_{k}-1))\cdots A(\theta +2\pi\frac{p_{k'}}{q_{k'}},1 )A(\theta,0) \\
    & -A(\theta+2(T_{k}-1)\omega,T(2(T_{k}-1))\cdots A(\theta +\omega,1 )A(\theta,0)||\\
  &=||A(\theta+4(T_{k}-1)\pi\frac{p_{k'}}{q_{k'}},T(2(T_{k}-1))\cdots A(\theta +2\pi\frac{p_{k'}}{q_{k'}},1 )A(\theta,0)\\
    &-A(\theta+4(T_{k}-1)\pi\frac{p_{k'}}{q_{k'}},T(2(T_{k}-1))\cdots A(\theta +\omega,1 )A(\theta,0)\\
    &+A(\theta+4(T_{k}-1)\pi\frac{p_{k'}}{q_{k'}},T(2(T_{k}-1))\cdots A(\theta +\omega,1 )A(\theta,0)\\
    &\cdots\\
    & -A(\theta+2(T_{k}-1)\omega,T(2(T_{k}-1))\cdots A(\theta +\omega,1 )A(\theta,0)||\\
    &=||A(\theta+4(T_{k}-1)\pi\frac{p_{k'}}{q_{k'}},T(2(T_{k}-1))\cdots (A(\theta +(2\pi\frac{p_{k'}}{q_{k'}}),1 )-A(\theta +\omega,1 ))A(\theta,0)\\
    &\cdots\\
   & (A(\theta+4(T_{k}-1)\pi\frac{p_{k'}}{q_{k'}},T(2(T_{k}-1))-A(\theta+2(T_{k}-1)\omega,T(2(T_{k}-1)))\cdots A(\theta +\omega,1 )A(\theta,0)||.
    \end{split}
\end{equation}
The main difference of equation of \ref{14} is

\begin{equation*}
\begin{split}
   & (A(\theta +(2\pi\frac{p_{k'}}{q_{k'}}),1 )-A(\theta +\omega,1 )) \\
    & \cdots\\
    &(A(\theta+4(T_{k}-1)\pi\frac{p_{k'}}{q_{k'}},T(2(T_{k}-1))-A(\theta+2(T_{k}-1)\omega,T(2(T_{k}-1))).
\end{split}
\end{equation*}
 This can be control by  \ref{16}, and the other term can be control by  $\bar{M}$

 So,there has:
\begin{equation*}
 \begin{split}
    & ||B_{k}^{T_{k}}(\theta )-B^{T_{k}}(\theta )|| \\
     & \leq 2T_{k} {\bar{M}}^{2T_{k}-1}4C\pi k^{-T_{k}},
\end{split}
\end{equation*}
and there are the similar difference of  the other term:

\begin{equation*}
\begin{split}
   & ||B_{k}^{2 T_{k}}(\theta )-
B^{2 T_{k}}(\theta )||, \\
    &||B_{k}^{-T_{k}}(\theta )-
B^{-T_{k}}(\theta )||,\\
&||B_{k}^{-2T_{k}}(\theta )-
 B^{-2T_{k}}(\theta )||.
\end{split}
\end{equation*}
Thus,
\begin{equation*}
\max_{a=\pm 1,\pm 2}\{||B(aT_{k})-B_{k}(aT_{k})||\}\rightarrow 0,
\end{equation*}
as
\begin{equation*}
 k\rightarrow 0.
\end{equation*}
Then for arbitrary norm $1$ vector  $v$,
\begin{equation*}
\max_{a=\pm 1,\pm 2}\{||B(aT_{k})||\}\geq\frac{1}{2}||v||\geq\frac{1}{2},
\end{equation*}
as
\begin{equation*}
\limsup_{n} \frac{||B(n) v||}{||v||}\geq\limsup_{n}\frac{\max_{a=\pm 1,\pm 2}\{||B(aT_{k})||\}}{||v||}\geq\frac{1}{4}.
\end{equation*}
So in the condition there can not have point spectrum.
And through the Lyapunove exponent is positive and continuous know that there can not have absolute continuous spectrum.

And the spectrum is not empty, so the spectrum of periodic coupling  AMO  model in the condition is purely singular
continuous.
\end{proof}

\section*{Acknowledgments}

We like to thank the anonymous referee for carefully reading the manuscript and providing  numerous helpful ramarks.

\bibliographystyle{IEEEtran}
\bibliography{111}

\end{document}